\documentclass{article}
\usepackage{maa-monthly}

\theoremstyle{theorem}
\newtheorem{theorem}{Theorem}

\theoremstyle{definition}

\newcommand\bigpar[1]{\bigl(#1\bigr)}
\newcommand\Bigpar[1]{\Bigl(#1\Bigr)}

\begin{document}

\title{Lagrange Inversion Formula by Induction}
\markright{Lagrange Inversion by Induction}
\author{Erlang Surya and Lutz Warnke}

\maketitle

\begin{abstract}
We present a simple inductive proof of the Lagrange Inversion Formula. 
\end{abstract}

\section{Introduction.}\label{sec:LIF}
The Lagrange inversion formula is a fundamental result in combinatorics. 
In its most basic form (see Theorem~\ref{thm:LIF:N} with ${H(z)=z}$ and~${H'(z)=1}$), it solves the functional equation~${A(x)=x\Phi(A(x))}$ for~${A(x)}$, by expressing the coefficients of the formal power series~$A(x)$ in terms of the coefficients of the formal power series~$\Phi(z)$. 
Functional equations of this form frequently arise in enumerative combinatorics, 
and in many applications the Lagrange inversion formula thus yields explicit counting formulas 
(e.g., for trees, permutations, 
and planar~maps); see~\cite{RandomTrees,Analytic,AofA,StanleyBook2,WilfBook}. 

As is so often the case for fundamental results, there are many different proofs of the \mbox{Lagrange} inversion formula, 
including ones based on Cauchy’s coefficient formula for holomorphic functions, residues of formal Laurent Series, and 
tree-counting arguments, just to name a few (see~\cite[Section~5.1]{WilfBook}, \cite[Section~4]{GesselSurvey}, \cite[Section~5.4]{StanleyBook2} and~\cite{BLL,Comtet,Egorychev,Henrici,Hofbauer,Krattenthaler} for more details and additional proofs). 
Furthermore, as one might expect, 
there are many different generalizations of the \mbox{Lagrange} inversion formula, including multivariate forms 
(see~\cite{Gessel1987,GesselSurvey,GouldenJackson,merlini2006,Sokal} and the references~therein).

In this expository note we present a simple and elementary 
\mbox{`just-do-it'} inductive proof of the \mbox{Lagrange} inversion formula 
(where all proof-steps emerge naturally).  
%
\begin{theorem}[Lagrange~Inversion~Formula]\label{thm:LIF:N}
Assume that~${A(x)=\sum_{n \ge 0}a_n x^n}$ and~${\Phi(z)=\sum_{r \ge 0}c_r z^r}$ are formal power series satisfying
\begin{equation}\label{as:LIF}
A(x)=x \Phi\bigpar{A(x)} .
\end{equation}
Then, for any integer~$n \ge 0$ and any formal power series~${H(z)=\sum_{r \ge 0} h_r z^r}$, 
\begin{equation}\label{eq:LIF}
n[x^n]H\bigpar{A(x)} = [z^{n-1}]H'(z)\Phi^n(z). 
\end{equation}
\end{theorem}
To fully understand and appreciate the statement and conclusion of Theorem~\ref{thm:LIF:N}, 
it might be useful to study the frequently asked questions discussed in the next section.

\section{Frequently asked questions.}\label{sec:faq}%
\subsection{What does the notation~$[x^n]F(x)$ mean?}
This is a widely-used~\cite{Bona,RandomTrees,Analytic,AofA,StanleyBook1,WilfBook} 
short-hand notation for the coefficient of~$x^n$ in the formal power series~$F(x)$,~i.e., 
\begin{align}\label{def:coeff}
[x^n]F(x) := f_n \quad \text{ when } \quad F(x)=\sum_{r \ge 0} f_r x^r.
\end{align}

\subsection{What are formal power series?}
In brief: given a commutative coefficient ring~$R$, the ring~$R[[x]]$ of formal power series 
is the set of all `formal sums' of the form~$\sum_{r \ge 0} c_r x^r$ with~$c_r \in R$, 
where addition, multiplication and differentiation are defined naturally:
\begin{align}
\label{def:add}
\sum_{r \ge 0}a_r x^r +\sum_{r \ge 0}b_r x^r & := \sum_{r \ge 0}(a_r+b_r) x^r ,\\
\label{def:mult}
\sum_{r \ge 0}a_r x^r \cdot \sum_{r \ge 0}b_r x^r & := \sum_{r \ge 0} \bigpar{\sum_{0 \le s \le r} a_{s} b_{r-s}} x^r,\\
\label{def:der}
\bigpar{\sum_{r \ge 0}a_r x^r}' & := \sum_{r \ge 1} (r a_r) x^{r-1} = \sum_{r \ge 0} a_r (x^{r})'.
\end{align}
The ring~$R[[x]]$ of formal power series satisfies (more or less) all the properties one would expect, including the following well-known derivative formula:
\begin{equation}\label{eq:der:identity}
    (\Phi^m(z))' = m \Phi'(z) \Phi^{m-1}(z) \qquad \text{for any integer~$m \ge 1$,}
\end{equation}
see~Section~\ref{sec:apx} for a short proof. 
For more details about formal power series we refer to the textbooks~\cite{GouldenJackson,StanleyBook1} 
or the \textit{American Mathematical Monthly} expository paper~\cite{Niven} 
(which won the Lester R.~Ford Award for expository~excellence); see also~\cite{merlini2007}.

\subsection{Why is the conclusion~\eqref{eq:LIF} useful in enumerative combinatorics?} 
In many enumeration problems, the formal power series~$A(x)$ is used as follows~\mbox{\cite{Bona,RandomTrees,Analytic,AofA,StanleyBook1,WilfBook}}: 
the coefficients~$a_n$ encode 
the number of objects of size~$n$, such as the number of certain \mbox{$n$-vertex} trees. 
Exploiting combinatorial properties of the objects of interest, one then infers a functional equation for~$A(x)$: 
for example, given an integer~$t \ge 1$, when~$a_n$ denotes the number of unlabelled rooted plane \mbox{$t$-ary} trees with~$n$ vertices (both external and internal) as in~\cite[Example~I.14 (p.~68)]{Analytic}, then one can obtain that
\begin{equation}\label{eq:example:0}
A(x)=x\bigpar{1+A^t(x)}.
\end{equation}
The Lagrange inversion formula shines when these functional equations cannot be explicitly solved for~$A(x)$, 
which in example~\eqref{eq:example:0} is the case when~${t \ge 5}$. 
In these situations the crux is that~\eqref{eq:LIF} still allows us to determine the coefficients of~$A(x)$: 
setting~$\Phi(z):=1+z^t$ and~$H(z):=z$ (so that~$H'(z)=1$), 
for~$n \ge 1$ it~directly~gives
\begin{equation}\label{eq:example:1}
a_n = [x^n]A(x) \stackrel{\eqref{eq:LIF}}{=}  \frac{1}{n}[z^{n-1}]\hspace{-0.6em}\underbrace{(1+z^t)^{n}}_{=\sum_{r=0}^{n}\tbinom{n}{r}z^{tr}} 
\hspace{-0.3em} = \frac{1}{n}\binom{n}{\tfrac{n-1}{t}} \quad \text{ provided $t \; | \; n-1$}.
\end{equation}
This illustrates the conceptual upshot of~\eqref{eq:LIF} in applications: 
 it can specify the coefficients of an unknown formal power series~$A(x)$ that is defined by the functional equation~\eqref{as:LIF} 
in terms of the known formal power series~$\Phi(z)$; see~\cite{RandomTrees,Analytic,AofA,StanleyBook2,WilfBook}.

\subsection{What is the point of arbitrary~$H(z)$ in~\eqref{eq:LIF}: is~$H(z)=z$ not enough?}
Building upon the previous question, in many enumeration problems the following idea is used: when a formal power series counts certain objects, then suitable functions of it count other objects of interest (see~\mbox{\cite[Chapter~3]{Bona}}, \mbox{\cite[Sections~I.2 and~II.2]{Analytic}}, or \mbox{\cite[Sections~5.2 and~5.3]{AofA}}). 
For example, if~$A(x)$ counts \mbox{$t$-ary} trees as in~\eqref{eq:example:0},~then
\begin{equation}\label{eq:example:2}
B(x):=A^k(x)
\end{equation}
counts so-called ordered \mbox{$k$-forests} of \mbox{$t$-ary} trees, 
which are simply \mbox{$k$-sequences} of \mbox{$t$-ary} trees; see~\cite[below~(69) on~p.~66]{Analytic}.
Having~$H(z)$ in~\eqref{eq:LIF} crucially allows us to directly determine the coefficients of~$B(x)$:
using~$H(z):=z^k$ and~$\Phi(z):=1+z^t$~gives
\[
b_n = [x^n]B(x) \stackrel{\eqref{eq:LIF}}{=} \frac{1}{n}[z^{n-1}]kz^{k-1} \Phi^n(z) = \frac{k}{n}[z^{n-k}](1+z^t)^{n} 
\]
for~$n \ge 1$, which can then be computed analogously to~\eqref{eq:example:1}. 
This illustrates why for applications it is useful to allow for arbitrary formal power series~$H(z)$ in~\eqref{eq:LIF}. 

\subsection{Doesn't the Lagrange~Inversion~Formula require~$a_0=0$ and~$c_0 \neq 0$?}
The two assumptions~$a_0=0$ and~$c_0 \neq 0$ arise naturally in applications of the Lagrange~Inversion~Formula. Indeed, it is well-known (and not difficult to check) that the assumed functional equation~\eqref{as:LIF} requires~${a_0=0}$, and, furthermore, that the special case~${c_0=0}$ corresponds to the degenerate power series~${A(x)=0}$. 
Nevertheless, it turns out that 
Theorem~\ref{thm:LIF:N} is true without these two assumptions, i.e., they are formally~redundant. 

\subsection{Isn't the Lagrange Inversion Formula about the compositional inverse?}
There are indeed 
formulations of the Lagrange Inversion Formula that concern the compositional inverse~$F^{\langle -1\rangle}(x)$ of a given formal power series~$F(x)=\sum_{r \ge 1}f_r x^r$, 
which satisfies~$F^{\langle -1\rangle}(F(x))=x=F(F^{\langle -1\rangle}(x))$. 
To this end we need to assume the existence of~$F^{\langle -1\rangle}(x)$, which turns out to be equivalent to~$f_1$ being invertible in the coefficient ring~$R$ (in which case $F^{\langle -1\rangle}(x)=\sum_{r \ge 1} g_r x^r$ with~$g_1=1/f_1$; see~\cite[Section~5.4]{StanleyBook2} and~\cite[Section~1.1]{GesselSurvey}).
%
%
To relate this setup to 
Theorem~\ref{thm:LIF:N}, we set~$\phi(x):=x/F(x)$ and observe that $x\phi(F^{\langle -1\rangle}(x))=F^{\langle -1\rangle}(x)$, 
so by invoking~\eqref{eq:LIF} with~$A(x):=F^{\langle -1\rangle}(x)$ it follows that, 
for any integer~$n \ge 0$ and any formal power series~${H(z)=\sum_{r \ge 0} h_r z^r}$, 
\begin{equation}\label{eq:LIF:compinverse}
n[x^n]H\bigl(F^{\langle -1\rangle}(x)\bigr)=[x^{n-1}] H'(x)\left( \frac{x}{F(x)} \right)^n . 
\end{equation}
The conceptual crux is that~\eqref{eq:LIF:compinverse} 
relates the coefficients of the formal power series~$F(x)$ and its compositional inverse~$F^{\langle -1\rangle}(x)$; 
for more details we refer  
to~\mbox{\cite[Section~3.8]{Comtet}}, \cite[Section~6.12]{AofA}, \cite[Section~5.4]{StanleyBook2} and the references~therein.

\section{Proof by Induction.}\label{sec:proof}
\begin{proof}[Proof of Theorem~\ref{thm:LIF:N}]
Using induction on~$n \ge 0$, for each~$n$ we shall prove that~\eqref{eq:LIF} holds for any formal power series~$H(z)$.  
The base case~$n=0$ is trivial, since both sides of~\eqref{eq:LIF} are zero (for the right-hand side the crux is that the power~$n-1$ of~$z$~is~negative).

We now turn to the induction step~$n \ge 1$, where we first exploit that the derivative is a linear operator: 
indeed, for the induction step it suffices to establish~that 
\begin{equation}\label{eq:LIF:goal:0}
n[x^n] A^k(x) = [z^{n-1}] (z^k)' \Phi^n(z)
\end{equation}
for all integers~$k \ge 0$, as the desired identity~\eqref{eq:LIF} then follows for any formal power series~$H(z)=\sum_{k \ge 0}h_k z^k$ using linearity of the~$[x^n]$ operator and~\eqref{def:der}:
\begin{equation*}
\begin{split}
n[x^n]H(A(x)) & =\sum_{k \ge 0}h_k n[x^n]A^k(x) \stackrel{\eqref{eq:LIF:goal:0}}{=} \sum_{k \ge 0}h_k [z^{n-1}](z^k)'\Phi^n(z) \\
& =  [z^{n-1}]\Bigpar{\sum_{k \ge 0}h_k(z^k)'}\Phi^n(z) \stackrel{\eqref{def:der}}{=} [z^{n-1}]H'(z)\Phi^n(z).
\end{split}
\end{equation*}
To complete the induction step, in view of~${(z^k)'=kz^{k-1}}$ it thus suffices to prove that, for all integers~$k \ge 0$,  
\begin{equation}\label{eq:LIF:goal}
n[x^n] A^k(x) = k[z^{n-k}] \Phi^n(z).
\end{equation}
It will be convenient (and instructive) to first verify~\eqref{eq:LIF:goal} in a few degenerate cases:  
\begin{itemize}
\item \emph{Case~$k=0$:} here both sides of~\eqref{eq:LIF:goal} are zero: 
for the left-hand side the crux is that~$A^k(x)=A^0(x)$ contains no powers of~$x$ of form~$x^n$ with~$n \ge 1$. 
\item \emph{Case~$k > n$:} here both sides of~\eqref{eq:LIF:goal} are again zero: 
for the right-hand side the crux is that the power~$n-k$ of~$z$ is negative, and for the left-hand side the crux is that the assumption~\eqref{as:LIF} implies that in~$A(x)^k$ all occurring powers of~$x$ are higher~than~$n$.
\item \emph{Case~$k = n$:} here~\eqref{eq:LIF:goal} is true since, using the assumption~\eqref{as:LIF} and~$n=k$, it readily follows that 
${n[x^n] A^k(x)} = {n[x^{0}]\Phi^n(A(x))} = {n (c_0)^n} = {k [z^{n-k}] \Phi^n(z)}$. 
\end{itemize}
It thus remains to verify~\eqref{eq:LIF:goal} in the case~$1 \le k < n$. 
Assumption~\eqref{as:LIF} implies~that
\begin{equation}\label{eq:LIFB:step1}
n[x^n] A^k(x) \stackrel{\eqref{as:LIF}}{=} n[x^n] x^k \Phi^k\bigpar{A(x)} = n [x^{n-k}] \Phi^k\bigpar{A(x)}. 
\end{equation}
By the induction hypothesis we may apply~\eqref{eq:LIF} with~${H(z)=\Phi^k(z)}$ and~$n$ replaced by~${n-k}$, 
and so by the derivative identity~\eqref{eq:der:identity} with~$m=k$ and~$m=n$ 
it follows~that 
\begin{equation}\label{eq:LIFB:step2}
\begin{split}
n [x^{n-k}] \Phi^k\bigpar{A(x)} & \stackrel{\eqref{eq:LIF}}{=} \frac{n}{n-k}[z^{n-k-1}] (\Phi^k(z))' \Phi^{n-k}(z)\\
& \stackrel{\eqref{eq:der:identity}}{=} \frac{k}{n-k} [z^{n-k-1}] n\Phi'(z)\Phi^{n-1}(z) \\
& \stackrel{\eqref{eq:der:identity}}{=} \frac{k}{n-k} [z^{n-k-1}] (\Phi^{n}(z))'.
\end{split}
\end{equation}
Finally, observe that by definition 
(see~\eqref{def:coeff} and~\eqref{def:der} above) 
we have 
\begin{equation}\label{eq:LIFB:step3}
[z^{n-k-1}] (\Phi^{n}(z))' \stackrel{\eqref{def:der}}{=}  (n-k) [z^{n-k}] \Phi^{n}(z),
\end{equation}
which together with~\eqref{eq:LIFB:step1} and~\eqref{eq:LIFB:step2} establishes the desired identity~\eqref{eq:LIF:goal}, completing the proof of the induction step (and thus Theorem~\ref{thm:LIF:N}). 
\end{proof}

\section{Discussion.} 
Let us now take a step back, and discuss the structure of our \mbox{inductive} proof of the Lagrange inversion formula. 
The first reduction step is standard (and intuitive): 
exploiting the linearity of the derivative, it suffices to prove the desired identity~\eqref{eq:LIF} for monomials~$H(z)=z^k$, i.e., it suffices to prove~\eqref{eq:LIF:goal:0}, which directly reduces to~\eqref{eq:LIF:goal}. 
In the induction step, it is natural to insert assumption~\eqref{as:LIF} to arrive at~\eqref{eq:LIFB:step1}, 
which is directly amenable to the induction hypothesis for a suitable formal power series~$H$ (exploiting that the induction hypothesis applies to arbitrary~$H$ instead of just monomials). 
The remaining steps from~\eqref{eq:LIFB:step2} onwards are again natural, 
and simply use the well-known derivative identities~\eqref{eq:der:identity} and~\eqref{eq:LIFB:step3}.
To sum up: all steps of the proof emerged naturally (since  none of them required any non-trivial ideas or insights), 
so we arguably presented a \mbox{`just-do-it'~proof} of the Lagrange~Inversion~Formula.

We remark that \cite[Section~4.2]{GesselSurvey} also contains an inductive proof, which has some similarities to the one given above.
However, that inductive proof is used to prove a somewhat indirect and less natural variant of the Lagrange inversion formula, so an additional argument is needed to deduce~\eqref{thm:LIF:N}. 
For combinatorial applications~\eqref{thm:LIF:N} is perhaps the most useful form of the Lagrange inversion formula, 
so it seems adequate to record our more direct (and simpler) inductive proof for expository reasons.

\section{Appendix.}\label{sec:apx}
We close by outlining, for completeness, a short proof of the well-known derivative identity~\eqref{eq:der:identity} for formal power series. 
First note~that~\eqref{def:der} gives 
\begin{equation}\label{eq:der:mon}
(z^{r}z^{s})' \stackrel{\eqref{def:der}}{=} (r+s)z^{r+s-1} \stackrel{\eqref{def:der}}{=} (z^{r})' z^{s} + z^{r} (z^{s})', 
\end{equation}
which by linearity implies (similar to the induction step in Section~\ref{sec:proof}) that
\begin{equation}\label{eq:der:pr}
\bigpar{F(z)G(z)}' = F'(z) G(z) + F(z) G'(z)
\end{equation}
for any two formal power series~$F(z)=\sum_{r \ge 0}f_rz^r$ and~$G(z)=\sum_{s \ge 0}g_sz^s$. 
Using the product rule~\eqref{eq:der:pr}, it then is easy to prove the desired derivative identity~\eqref{eq:der:identity} by induction on~$m \ge 1$ (the base case~$m=1$ being trivial).

\begin{acknowledgment}{Acknowledgment.}
We are very grateful to \mbox{Ira~Gessel} for several helpful comments and simplifications.
We thank \mbox{Juanjo~Ru\'e} for suggesting the \mbox{$t$-ary} tree example~\eqref{eq:example:0}, and  the referees for pointing out additional references. 
This work was supported by NSF~CAREER grant~DMS-2225631 and a Sloan Research Fellowship. 
\end{acknowledgment}


\begin{biog}

\item[Erlang Surya] 
\begin{affil}
Department of Mathematics, University of California San Diego, La Jolla CA~92093, USA\\
esurya@ucsd.edu
\end{affil}

\item[Lutz Warnke] 
\begin{affil}
Department of Mathematics, University of California San Diego, La Jolla CA~92093, USA\\
lwarnke@ucsd.edu
\end{affil}
\end{biog}
\vfill\eject

\end{document}